\documentclass[12pt]{amsart}

\usepackage{fullpage,url,amssymb,enumerate,colonequals}
\usepackage[all]{xy} 

\usepackage{color}

\newcommand{\Aff}{\mathbb{A}}
\newcommand{\F}{\mathbb{F}}
\newcommand{\PP}{\mathbb{P}}
\newcommand{\Q}{\mathbb{Q}}
\newcommand{\Z}{\mathbb{Z}}

\newcommand{\calE}{\mathcal{E}}

\DeclareMathOperator{\Jac}{Jac}
\DeclareMathOperator{\rank}{rank}
\DeclareMathOperator{\rk}{rk}
\DeclareMathOperator{\Sel}{Sel}

\newcommand{\To}{\longrightarrow}

\numberwithin{equation}{section}

\newtheorem{theorem}[equation]{Theorem}
\newtheorem{lemma}[equation]{Lemma}
\newtheorem{corollary}[equation]{Corollary}
\newtheorem{proposition}[equation]{Proposition}

\theoremstyle{definition}
\newtheorem{remark}[equation]{Remark}

\definecolor{darkgreen}{rgb}{0,0.5,0}
\usepackage[
        colorlinks, citecolor=darkgreen,
        backref,
        pdfauthor={Michael Stoll},
        pdftitle={Diagonal genus 5 curves, elliptic curves over Q(t), and rational diophantine quintuples},
]{hyperref}
\usepackage[alphabetic,backrefs,lite]{amsrefs} 

\begin{document}

\title[Rational diophantine quintuples]%
      {Diagonal genus~5 curves, elliptic curves over $\Q(t)$, \\ and rational diophantine quintuples}

\author{Michael Stoll}
\address{Mathematisches Institut,
         Universit\"at Bayreuth,
         95440 Bayreuth, Germany.}
\email{Michael.Stoll@uni-bayreuth.de}
\urladdr{http://www.computeralgebra.uni-bayreuth.de}

\date{\today}

\begin{abstract}
  The problem of finding all possible extensions of a given rational diophantine quadruple
  to a rational diophantine quintuple is equivalent to the determination of the set
  of rational points on a certain curve of genus~5 that can be written as an intersection
  of three diagonal quadrics in~$\PP^4$. We discuss how one can (try to) determine the
  set of rational points on such a curve. We apply our approach to the original question
  in several cases. In particular, we show that Fermat's diophantine quadruple $(1,3,8,120)$
  can be extended to a rational diophantine quintuple in only one way, namely by~$777480/8288641$.

  We then discuss a method that allows us to find the Mordell-Weil group of an elliptic
  curve~$E$ defined over the rational function field~$\Q(t)$ when $E$ has full $\Q(t)$-rational
  $2$-torsion. This builds on recent results of Dujella, Gusi\'c and Tadi\'c. We give several
  concrete examples to which this method can be applied. One of these results implies that there is only
  one extension of the diophantine quadruple $\bigl(t-1,t+1,4t,4t(4t^2-1)\bigr)$ over~$\Q(t)$.
\end{abstract}

\maketitle


\section{Introduction}

This paper is mainly concerned with two computational questions in arithmetic
algebraic geometry. The first is how one can attempt to find the set of rational
points on a `diagonal' curve~$C$ of genus~$5$ over~$\Q$. i.e., a curve obtained
as the smooth intersection of three diagonal quadrics in~$\PP^4$. This problem has been
studied by Gonz\'alez-Jim\'enez and Xarles~\cites{GJX2014,GJ2015}. Our contribution
here is to provide a more conceptual geometric framework for their approach and to
suggest some improvements to the resulting algorithm. This is done in
Section~\ref{S:genus5}; the algorithm is given at the end of that section.
Some important features that we
use are the splitting of the Jacobian~$J$ of~$C$ as a product of five elliptic curves
over~$\Q$, up to isogeny, and the fact that $J$ has a large rational $2$-torsion
subgroup, which allows us to consider many \'etale double coverings of~$C$.
The Prym varieties of these coverings are isogenous (over~$\bar{\Q}$) to a product
of four elliptic curves defined (in general) over a biquadratic number field. This allows us
to set up various ways of applying `Elliptic Curve Chabauty'~\cite{Bruin2003} to our situation.

The second question is how one can find generators of the Mordell-Weil group
of an elliptic curve~$E$ over the rational function field~$\Q(t)$, in the case when
$E$ has full $\Q(t)$-rational $2$-torsion (or at least one $\Q(t)$-rational point of order~$2$).
Here we build on earlier work by Gusi\'c and~Tadi\'c~\cites{GusicTadic2012,GusicTadic2015},
which in turn uses an idea of Dujella~\cite{Dujella2000}. This is done
in Section~\ref{S:EQt}; the resulting algorithm is given at the end of that section.

We then apply our algorithms to questions related to rational diophantine $m$-tuples.
A \emph{diophantine $m$-tuple} is an $m$-tuple $(a_1, \ldots, a_m)$ of distinct
nonzero integers such that $a_i a_j + 1$ is a square for all $1 \le i < j \le m$.
A \emph{rational diophantine $m$-tuple} is an $m$-tuple of distinct nonzero rational
numbers with the same property. A recent result by He, Togb\'e and
Ziegler~\cite{HeTogbeZiegler} establishes that no diophantine quintuple exists.
On the other hand, it has been shown that there are infinitely many rational diophantine
sextuples~\cite{DKMS2017}, but no rational diophantine septuples are known.
See Andrej Dujella's ``Diophantine $m$-tuples'' page~\cite{Dujella} for
more information on background and results regarding diophantine $m$-tuples.

Given a (rational) diophantine quadruple $(a_1, a_2, a_3, a_4)$, we can ask
in how many ways it can be extended to a (rational) diophantine quintuple
by adding a number $a_5 \notin \{0, a_1, a_2, a_3, a_4\}$ such that
\[ a_i a_5 + 1 \quad \text{is a square} \quad \text{for all $1 \le i \le 4$.} \]
Replacing $a_5$ by~$x$, we obtain the system of equations
\[ a_1 x + 1 = u_1^2\,, \quad a_2 x + 1 = u_2^2\,, \quad a_3 x + 1 = u_3^2\,, \quad a_4 x + 1 = u_4^2 \]
describing a curve in~$\Aff^5$, whose integral (or rational) points with
$x \neq 0, a_1, a_2, a_3, a_4$ correspond to the solutions of our problem
(more precisely, we obtain $16$~points for each solution, depending on how we
choose the signs of the square roots~$u_i$). This curve is irreducible and smooth;
it has genus~$5$. Eliminating~$x$, it is given by the three quadrics
\[ a_i u_4^2 - a_4 u_i^2 - (a_i - a_4) = 0\,, \quad i = 1,2,3 \,; \]
by homogenizing these equations we obtain a description of the smooth projective
model~$C$ of the curve as an intersection of three diagonal quadrics in~$\PP^4$.
We call such a curve a \emph{diagonal curve of genus~$5$}. Note that by Faltings'
theorem~\cite{Faltings1983}, $C$ has only finitely many rational points, hence our original problem
has only finitely many solutions for any given quadruple.
We can then use the algorithm of Section~\ref{S:genus5} to try
to determine the set $C(\Q)$ of rational points on~$C$.
We apply this general approach to the extension problem of diophantine
quadruples. More precisely, we can show that a number of quadruples from the family
\[ \bigl(t-1, t+1, 4t, 4t(4t^2-1)\bigr) \]
can be extended in exactly one way (the ``regular'' extension, which exists for
every quadruple).

This prompts the question whether the regular extension is the only extension
that exists generically for quadruples in the family above. To answer this question,
we have to study the situation over the rational function field~$\Q(t)$.
It turns out that it is sufficient to determine the group~$E(\Q(t))$
for a specific elliptic curve~$E$ over~$\Q(t)$.
This is where we use the algorithm developed in Section~\ref{S:EQt}.
The result allows us to show that
the regular extension is the only extension to a diophantine quintuple over~$\Q(t)$.
We also show that $J(\Q(t))$ is generated by the differences of the $32$ $\Q(t)$-rational
points on the associated genus~$5$ curve~$C$ (where $J$ is again the Jacobian of~$C$)
and determine the structure of the group.

\subsection*{Acknowledgments}
I would like to thank Andrej Dujella for the motivating question that led to this
work and also for numerous helpful exchanges related to it. I also thank
Ivica Gusi\'c for pointing out various minor errors in earlier versions of this paper
and an anonymous referee for some useful comments.
The computations were done using the Magma Computer Algebra System~\cite{Magma}.


\section{Diagonal genus 5 curves} \label{S:genus5}

Let $K$ be a field not of characteristic~$2$.
The canonical model of a non-hyperelliptic curve~$C$ of genus~5 over~$K$ is an
intersection of three quadrics in~$\PP^4$. If the coordinates on~$\PP^4$
can be chosen so that the quadrics are diagonal, then we say that
$C$ is a \emph{diagonal genus~5 curve}. We denote suitable coordinates
by $u_0, \ldots, u_4$. Assume that we can define $C$ in~$\PP^4$ by
\[ \sum_{j=0}^4 a_j u_j^2 = \sum_{j=0}^4 b_j u_j^2 = \sum_{j=0}^4 c_j u_j^2 = 0 \]
and let $M$ be the $3 \times 5$ matrix whose rows are $(a_0, \ldots, a_4)$,
$(b_0, \ldots, b_4)$, $(c_0, \ldots, c_4)$. Then the condition for
the curve defined in this way to be smooth (and hence of genus~$5$) is
that no $3 \times 3$ minor of~$M$ vanishes. This is equivalent to saying
that the net of quadrics generated by the three quadrics above
does not contain quadrics of rank~$2$ or less. There are then exactly
ten quadrics of rank~$3$, which are obtained by eliminating two of
the variables.

Any such curve~$C$ has a subgroup~$A$
isomorphic to $(\Z/2\Z)^4$ in its automorphism group, which is generated
by the five involutions~$\sigma_j$ changing the sign of~$u_j$ and
leaving the other coordinates fixed; their product is the identity.

\begin{remark}
  The (coarse) moduli space of diagonal genus~$5$ curves has dimension~$2$:
  Fixing the coordinates, any such curve is given by a $3$-dimensional space
  of diagonal quadratic forms in five variables. The non-vanishing of the
  $3 \times 3$-minors implies that we can write the equations in the form
  \[ u_2^2 = a u_0^2 + b u_1^2, \quad
     u_3^2 = a' u_0^2 + b' u_1^2, \quad
     u_4^2 = a'' u_0^2 + b'' u_1^2,
  \]
  where the vectors $(a, b)$, $(a', b')$ and~$(a'', b'')$ are linearly independent
  in pairs. Over an algebraically closed field, we can then scale the variables
  to make $a = a' = a'' = b = 1$; we are left with two parameters $b', b''$
  that are different from $0$, $1$ and from each other. Permuting the variables
  induces an action of the symmetric group~$S_5$ on the open subset of~$\PP^2$
  given by the pairs $(b', b'')$; the quotient is the moduli space.

  We note that this is consistent with the expectation derived from the
  observation that the locus of Jacobians of diagonal genus~$5$ curves in
  the moduli space of principally polarized abelian $5$-folds of dimension~$15$
  is contained in the intersection of the space of Jacobians of dimension~$12$
  and the locus of abelian varieties isogenous to a product of five elliptic
  curves of dimension~$5$.

  Over a non-algebraically closed field~$k$, we have to take into account the
  twists coming from the action of~$A$, which leads to a parameterization
  by a subset of $k^2 \times (k^\times/k^{\times 2})^4$ (modulo the action of~$S_5$).
\end{remark}

Eliminating~$u_j$ leads to a double cover $\pi_j \colon C \to F_j$, where $F_j$ is
a curve of genus~$1$ given as the intersection of two diagonal quadrics
in~$\PP^3$. Write $E_j$ for the Jacobian elliptic curve of~$F_j$;
then $E_j$ has all its $2$-torsion points defined over~$K$.
(This is because $F_j$ is an intersection of diagonal quadrics,
which implies that the four singular quadrics in the pencil are
all defined over the ground field.)
We write $J$ for the Jacobian variety of~$C$. The maps $C \to F_j$ induce
homomorphisms $\pi_{j,*} \colon J \to E_j$, which we can combine into an isogeny
$\varphi \colon J \to \prod_{j=0}^4 E_j$ (see for example~\cite{Bremner1997}).
Pulling back divisors under~$\pi_j$ induces $\pi_j^* \colon E_j \to J$ such that
the composition $\pi_{j,*} \circ \pi_j^*$ is multiplication by~$2$ on~$E_j$.
The compositions $\pi_{j,*} \circ \pi_i^*$ for $i \neq j$ must be
zero, since generically the elliptic curves $E_i$ and~$E_j$ are not
isogenous. Combining the $\pi_j^*$ into
$\hat{\varphi} \colon \prod_{j=0}^4 E_j \to J$, $(P_j)_j \mapsto \sum_j \pi_j^*(P_j)$,
then gives an isogeny in the other direction such that $\varphi \circ \hat{\varphi}$
is multiplication by~$2$ on the product of the~$E_j$; this implies that
$\hat{\varphi} \circ \varphi$ is multiplication by~$2$ on~$J$.

We can describe part of the $2$-torsion subgroup~$J[2]$ in the following
way. Eliminating $u_i$ and~$u_j$ results in a conic~$Q_{ij}$; we have
the following commutative diagram:
\[ \xymatrix{ & C \ar[dl]_{\pi_i} \ar[dr]^{\pi_j} \\
              F_i \ar[dr] & & F_j \ar[dl] \\
              & Q_{ij}
            }
\]
Pulling back any point on~$Q_{ij}$ results in an effective divisor~$D_{ij}$ of
degree~$4$ on~$C$ such that $2D_{ij}$ is in the hyperplane class.
(If the point on~$Q_{ij}$ is the image of $P \in C$, then
$D_{ij} = (P) + (\sigma_i(P)) + (\sigma_j(P)) + (\sigma_i \sigma_j(P))$.)
The class of~$D_{ij}$ does not depend on the chosen point, since all
points on~$Q_{ij}$ are linearly equivalent. So any difference
$T_{\{\{i,j\},\{i',j'\}\}} = [D_{ij} - D_{i'j'}] \in J$ is in~$J[2]$.
Taking $[D_{01}]$ (say) as a base-point, we obtain nine generators
of the form $T_{ij} = T_{\{\{0,1\},\{i,j\}\}}$ (with $\{i,j\} \neq \{0,1\}$)
for our subgroup. There are some relations among these points, though.
On~$F_i$, the sum of the four divisors obtained from the~$Q_{ij}$
with $j \neq i$ is in twice the hyperplane class, which implies that
\[ \sum_{j \neq i} T_{ij} = 0 \qquad \text{for each $i \in \{0,\ldots,4\}$.} \]
Since four of these relations imply the fifth, we get a subgroup~$G$ of~$J[2]$ whose
$\F_2$-dimension is at most~$5$. Checking a concrete example shows
that~$G$ is indeed of that dimension.
One can check that $G \subset \ker(\varphi)$. Conversely, one can check
that $\bigoplus_j E_j[2]$ maps surjectively onto~$G$ under~$\hat{\varphi}$.
This implies that the kernel of~$\hat{\varphi}$ is of dimension~$5$
and that $G = \ker{\varphi}$.

Since the class of a point on~$Q_{ij}$
is defined over~$K$, so is the class of~$D_{ij}$, and therefore $G$ consists
of $K$-rational points. If $C(K)$ is non-empty
or if $K$ is a number field and $C$ has points over all completions of~$K$,
then all the conics have $K$-points, and  we can choose the divisors~$D_{ij}$
to be defined over~$K$.

Now assume that $K = \Q$ and that $P_0 \in C(\Q)$. Usually the differences
of the images under~$\pi_i$ of the points in the orbit of~$P_0$ under~$A$
will generate a group of rank~$1$ (changing two signs of~$u_j$ with $j \neq i$
corresponds to adding a $2$-torsion point, but changing one sign and
taking the difference with~$\pi_i(P_0)$ will usually give a point of
infinite order). So the rank of each~$E_i(\Q)$ will be positive, and the
rank of~$J(\Q)$ will be at least~$5$; in particular, standard Chabauty
techniques are not applicable, since no factor of~$J$ will have rank
strictly less than its dimension. The purpose of the following is to
describe an approach that can possibly succeed in determining~$C(\Q)$.
It is based on covering collections and `Elliptic Curve Chabauty'~\cite{Bruin2003}.
For general information on approaches for determining the set of
rational points on a curve, see also the survey~\cite{Stoll2011}.
In the case at hand, what we do follows the ideas underlying recent
work by Gonz\'alez-Jim\'enez and Xarles~\cites{GJX2014,GJ2015};
here we provide a more conceptual geometric framework.

Since by assumption $P_0$ is a rational point on~$C$, we have rational
points $\pi_j(P_0)$ on the curves~$F_j$. We then obtain
isomorphisms $F_j \stackrel{\simeq}{\to} E_j$ sending $\pi_j(P_0)$ to the origin.
In the following we will identify $F_j$ and~$E_j$ via this isomorphism.
Let $\iota \colon C \to J$ denote the embedding sending $P$ to the
class of $(P) - (P_0)$; since $P_0$ is rational, $\iota$ is defined over~$\Q$.
We can then consider the composition~$\delta$ of maps
\[ C(\Q) \stackrel{\iota}{\to} J(\Q) \stackrel{\varphi}{\to} \bigoplus_{j=0}^4 E_j(\Q)
         \to \bigoplus_{j=0}^4 \Sel^{(2)}(E_j/\Q)
         \to \bigoplus_{j=0}^4
            \Bigl(\frac{\Q^\times}{\Q^{\times 2}}
                   \oplus \frac{\Q^\times}{\Q^{\times 2}}\Bigr) =: H .
\]
Here $\Sel^{(2)}(E_j/\Q)$ denotes the $2$-Selmer group of~$E_j$ over~$\Q$.
The last map is the direct sum of the usual maps from $2$-descent on an
elliptic curve with full rational $2$-torsion. For any place~$v$ of~$\Q$,
there is a similar map, and we obtain a commutative diagram
\[ \xymatrix{C(\Q) \ar[r]^{\delta} \ar[d] & H \ar[d]^{\rho_v} \\
             C(\Q_v) \ar[r]^{\delta_v} & H_v
            }
\]
where $H_v$ is defined like~$H$, but replacing $\Q$ with~$\Q_v$, and the vertical
maps are the natural ones.

We assume that we have determined generators of each group~$E_j(\Q)$, or
at least generators of a subgroup of finite odd index, so that we know
the image~$H_0$ of $\bigoplus_j E_j(\Q)/2 E_j(\Q)$ in~$H$. Alternatively,
it may be sufficient to know the $2$-Selmer groups of the~$E_j$, which
gives us an upper bound~$H_0$ for the image of $\bigoplus_j E_j(\Q)/2 E_j(\Q)$ in~$H$.
The map~$\delta$
and its local equivalents~$\delta_v$ are given by evaluating certain
rational functions on~$C$ that are defined over~$\Q$ at a point~$P$
and then taking the square class of the result. So $\delta_v$ is locally
constant in the $v$-adic topology, which allows us to determine its
image with a finite computation for each given place~$v$.
We do this for a number of places and define~$H'_0$
to be the subset of~$H_0$ consisting of elements mapping into the
image of~$\delta_v$ for all places~$v$ that were considered.
Then $\delta(C(\Q)) \subset H'_0$, and if we are lucky, we even may
have equality, since we possibly know enough points in~$C(\Q)$ to verify that
$\delta$ surjects onto~$H'_0$.
In practice, it appears that using a few primes~$p$ of good reduction
(where the computation can be essentially done over~$\F_p$)
suffices to cut $H'_0$ down to the image of the known rational
points on~$C$. (What we do here is to compute
(an upper bound for) the Selmer set of~$C$
with respect to~$\hat{\varphi}$; compare~\cite{Stoll2011}
or~\cite{BS2009}. Note that $\Sel^{(\hat{\varphi})}(J/\Q)$ injects
into the direct sum of the $2$-Selmer groups of the~$E_j$ via the
map induced by~$\varphi$, since the elliptic curves have full rational
$2$-torsion and so $\bigoplus_j E_j(\Q)[2]$ surjects onto~$J(\Q)[\varphi] = G$.)

The action of~$A$ on~$C$ induces an action on the product of the~$E_j$
that on each~$E_j$ separately is generated by the translations by
$2$-torsion points and a map of the form $Q \mapsto P - Q$, where
$P = \pi_j(\sigma_i(P_0))$ for some fixed $i \neq j$. This translates
into the action of a subgroup~$H_1$ of~$H_0$ (the subgroup generated by
the images of the points in the $A$-orbit of~$P_0$) on~$H_0$ by translation;
the subset~$H'_0$ is a union of cosets of this subgroup. In the computation,
it will be sufficient to consider one representative of each coset, since
all results can be transported by the $A$-action to everything in the
same orbit.

Each nonzero element $T \in G$ induces an isogeny $J'_T \to J$ of degree~$2$;
we obtain~$J'_T$ by dividing the product of the~$E_j$ by the subgroup
of~$\ker \hat{\varphi}$ orthogonal to~$T$ under the Weil pairing
between $\ker \varphi$ and~$\ker \hat{\varphi}$.
Pulling back the isogeny under~$\iota$, we get an \'etale double cover $D_T \to C$.
For each $\xi \in H'_0$ we obtain a twist $D_{T,\xi} \to C$ of this
double cover with the property that all points in~$C(\Q)$ whose image
in~$H'_0$ is~$\xi$ lift to rational points on~$D_{T,\xi}$.
Then $D_{T,\xi}$ maps into the Prym variety of the double cover
$D_{T,\xi} \to C$, which is an abelian variety of dimension~$4$.

One can check that all but one of the nonzero elements~$T \in G$ can be represented as a
difference of two divisors of the form~$D_{ij}$. If
\[ T = [D_{ij} - D_{kl}] = [D_{ij} + D_{kl}] - L \,, \]
where $L$ denotes the class of a hyperplane section, then we obtain a model
of~$D_{T,\xi}$ in the following way. Fix points $P_{ij} \in Q_{ij}$ and
$P_{kl} \in Q_{kl}$, and let $\ell_{ij}(u_0, \ldots, u_4)$ and $\ell_{kl}(u_0, \ldots. u_4)$
be linear forms describing the tangent lines to the conics at the points
(note that $\ell_{ij}$ will not involve $u_i$ and~$u_j$; similarly for~$\ell_{kl}$).
Then the double covering of~$C$ given by adding the equation
\[ \gamma w^2 = \ell_{ij}(u_0, \ldots, u_4) \ell_{kl}(u_0, \ldots, u_4) \]
will be a (singular) model of~$D_{T,\xi}$, where $\gamma \neq 0$ has
to be chosen so that points mapping to~$\xi$ will lift. (Concretely,
if $P = (\upsilon_0 : \ldots : \upsilon_4) \in C(\Q)$ maps to~$\xi$, then we can
take $\gamma = \ell_{ij}(\upsilon_0, \ldots, \upsilon_4) \ell_{kl}(\upsilon_0, \ldots, \upsilon_4)$
if this value is nonzero.) If there is a rational point on~$C$
and therefore on all the conics, then we can take $P_{ij}$ and~$P_{kl}$
to be rational, and we obtain equations over~$\Q$.

For such~$T$, we can identify the Prym variety up to isogeny as the Weil restriction
of an elliptic curve defined over an \'etale algebra of degree~$4$
over~$\Q$ (which generically is a biquadratic number field).
This can be seen by choosing the points $P_{ij}$ and~$P_{kl}$ in such a way
that the linear forms~$\ell_{ij}$ and~$\ell_{kl}$ together involve only three
of the five variables. There are two cases.
\begin{enumerate}[1.]
  \item $T = [D_{ij} - D_{kl}]$ with $i,j,k,l$ distinct. Let $m$ be
        such that $\{i,j,k,l,m\} = \{0,\ldots,4\}$. We can write
        equations for the conics $Q_{ij}$, $Q_{kl}$ and~$Q_{ik}$ as
        \[   a u_k^2 = u_m^2 - b u_l^2\,, \quad
            a' u_i^2 = u_m^2 - b' u_j^2\,, \quad
           a'' u_l^2 = u_m^2 - b'' u_j^2\,.
        \]
        The intersection of the hyperplane $u_m = \sqrt{b} u_l$ with~$C$
        is twice a divisor of the form~$D_{ij}$ and the intersection of the
        hyperplane $u_m = \sqrt{b'} u_j$ is twice~$D_{kl}$. (This corresponds
        to taking $P_{ij}$ to be a point with $u_k = 0$ and $P_{kl}$ with $u_i = 0$.)
        Let $P \in C(\Q)$ be a point whose image in~$H'_0$ is~$\xi$, say
        $P = (\upsilon_0 : \ldots : \upsilon_4)$. Then $D_{T,\xi}$ maps to
        the curve $F_{T,\xi}$ given by
        \[ a'' u_l^2 = u_m^2 - b'' u_j^2\,, \quad
           \gamma w^2 = (u_m - \sqrt{b} u_l) (u_m - \sqrt{b'} u_j)
        \]
        in~$\PP^3$, where
        $\gamma = (\upsilon_m - \sqrt{b} \upsilon_l) (\upsilon_m - \sqrt{b'} \upsilon_j)$.
        This is a curve of genus~$1$ defined over $K = \Q(\sqrt{b}, \sqrt{b'})$.
        Let $E_{T,\xi}$ be its Jacobian elliptic curve.
        (Since $F_{T,\xi}$ has a point over~$K$ coming from~$P$, it is actually an elliptic
        curve itself.) Then Picard functoriality induces a homomorphism
        $R_{K/\Q} E_{T,\xi} \to \Jac(D_{T,\xi})$ that is an isogeny onto the Prym variety.
  \item $T = [D_{ij} - D_{ik}]$ with $i,j,k$ distinct.
        Let $l,m$ be such that $\{i,j,k,l,m\} = \{0,\ldots,4\}$. We can write
        equations for the conics $Q_{jk}$, $Q_{ik}$ and~$Q_{ij}$ as
        \[   a u_i^2 = u_l^2 - b u_m^2\,, \quad
            a' u_j^2 = u_l^2 - b' u_m^2\,, \quad
           a'' u_k^2 = u_l^2 - b'' u_m^2\,.
        \]
        The intersection of the hyperplane $u_l = \sqrt{b'} u_m$ with~$C$
        is twice a divisor of the form~$D_{ik}$ and the intersection of the
        hyperplane $u_l = \sqrt{b''} u_m$ is twice~$D_{ij}$. Let $P \in C(\Q)$
        be a point whose image in~$H'_0$ is~$\xi$, say
        $P = (\upsilon_0 : \ldots : \upsilon_4)$. Then $D_{T,\xi}$ maps to
        the curve $F_{T,\xi}$ given by
        \[ a u_i^2 = u_l^2 - b u_m^2\,, \quad
           \gamma w^2 = (u_l - \sqrt{b'} u_m) (u_l - \sqrt{b''} u_m)
        \]
        in~$\PP^3$, where
        $\gamma = (\upsilon_l - \sqrt{b'} \upsilon_m) (\upsilon_l - \sqrt{b''} \upsilon_m)$.
        This is a curve of genus~$1$ defined over $K = \Q(\sqrt{b'}, \sqrt{b''})$.
        Let $E_{T,\xi}$ be its Jacobian elliptic curve.
        (Since $F_{T,\xi}$ has a point over~$K$ coming from~$P$, it is actually an elliptic
        curve itself.) Then Picard functoriality induces a homomorphism
        $R_{K/\Q} E_{T,\xi} \to \Jac(D_{T,\xi})$ that is an isogeny onto the Prym variety.
\end{enumerate}

Each of the two cases covers $15$ possibilities for~$T$. Note that for each~$T$
there can be several ways of writing it as a difference of two divisors~$D_{ij}$,
and for each such representation there can be several ways of writing down
a curve~$F_{T,\xi}$. They will all lead to isogenous Weil restrictions, however.

More precisely, the points~$T$ occurring in Case~1 have a single representation
of the form $T = [D_{ij} + D_{kl}] - L$ in the sense that $\{\{i,j\}, \{k,l\}\}$
is uniquely determined. Since the construction above depends on choosing a representative
of each of $\{i,j\}$ and~$\{k,l\}$, we get four different curves $F_{T,\xi}$.
The points~$T$ occurring in Case~2 have two different representations as
$T = [D_{ij} + D_{ik}] - L = [D_{il} + D_{im}] - L$ (this comes from one of the
relations between the $[D_{ij}]$ mentioned earlier), but each representation gives
rise to only one curve~$F_{T,\xi}$.

From a computational point of view, this has the advantage that we can pick
the representation that involves the friendliest field~$K$ (with smallest
discriminant, say) or else gives us more possibilities for computing
rank bounds via Selmer groups. It turns out that we actually do get different
rank bounds in general, so it makes sense to look at all of the $90$~possibilities
obtained by making all possible choices, together with the curves $2$-isogenous to them
(generically, each of the elliptic curves obtained has one point of order~$2$
defined over its base field).

\begin{remark}
  Note that in the first case, we usually obtain elliptic curves over four
  different biquadratic fields whose Weil restrictions are all isogenous over~$\Q$
  to the Prym variety of the double cover.
  Taking the $2$-isogenous curves in two out of the four cases, we obtain
  four elliptic curves that become all isomorphic over the compositum~$K'$
  of their various fields of definition ($K'$ is generated by four square roots over~$\Q$).
  Let $E$ be this curve over~$K'$. Then $E$ is isomorphic over~$K'$ to all its conjugates,
  and the Weil restrictions are all isomorphic over~$K'$ to~$E^4$.
  The Weil restrictions are all isogenous over~$\Q$, but (in general) not isomorphic
  (as can be seen by considering their $2$-torsion subgroups,
  whose elements are defined over different octic fields).
\end{remark}

We now observe that we have, in each case, an elliptic curve~$F_{T,\xi}$
together with morphisms $D_{T,\xi} \to F_{T,\xi} \to \PP^1$ over~$K$,
where the morphism to~$\PP^1$ is given by the quotient of any two of the
$u$~coordinates involved in~$F_{T,\xi}$, such that the image of any point
$P \in D_{T,\xi}(\Q)$ in~$F_{T,\xi}(K)$ maps into~$\PP^1(\Q)$.

This is the setting
for the Elliptic Curve Chabauty method~\cite{Bruin2003}, which allows
us to find the set of such points in~$F_{T,\xi}(K)$ when the rank of this
Mordell-Weil group is strictly less than $[K : \Q] = 4$.

If the degree of~$K$ is less than~$4$, then
instead of one elliptic curve over a quartic field, we have to work
with two elliptic curves over a quadratic field (this is what is done
in the applications in~\cites{GJX2014,GJ2015}) or with four elliptic curves
over~$\Q$.

We therefore obtain the following procedure that may determine~$C(\Q)$.
For simplicity, we assume that we are in the generic case where all
fields~$K$ are of degree~$4$.

\begin{enumerate}[1.]
  \item For each $j \in \{0,\ldots,4\}$, compute the $2$-Selmer group of~$E_j$
        (or even, if possible, the group $E_j(\Q)/2 E_j(\Q)$); \\
        determine the group~$H_0$ and its subgroup~$H_1$.
  \item For a suitable finite set of places~$v$ of~$\Q$, determine
        $\delta_v(C(\Q_v)) \subset H_v$; \\
        use this to compute $H'_0 \subset H_0$. \\
        Verify that the set of known points in~$C(\Q)$ surjects onto~$H'_0$.
  \item Let $X$ be a set of representatives of $H'_0/H_1$.
  \item For each $\xi \in X$, do the following.
        \begin{enumerate}[a.]
          \item For each $T \in G$ as in cases 1 or~2 above,
                determine an upper bound for the rank~$r$ of $F_{T,\xi}(K)$
                (where $K$ is the quartic field as in Case 1 or~2). \\
                If $r \ge 4$ for all such~$T$, report failure and stop.
          \item For some $T$ such that $r \le 3$, perform
                the Elliptic Curve Chabauty computation to find all
                points $P' \in F_{T,\xi}(K)$ whose image in~$\PP^1$
                is rational.
          \item For each $P'$ obtained in this way, check if it lifts
                to a rational point~$P$ on~$C$. Collect all points
                found in this way in a set~$S$.
        \end{enumerate}
  \item Return $S$; it is a set of orbit representatives of the action
        of~$A$ on~$C(\Q)$.
\end{enumerate}

In Step~4a, we can for example compute the $2$-Selmer group of~$F_{T,\xi}$.
Note that the Jacobian of each~$F_{T,\xi}$ has a $K$-rational point of
order~$2$, so we can also compute the $2$-Selmer group of the $2$-isogenous
curve, which in some cases gives a better bound.

To perform Step~4b, we need to find generators of a finite-index subgroup
of~$F_{T,\xi}(K)$. We may get some points of infinite order from the
known points in~$C(\Q)$; if this does not give enough points, then it may
be difficult to find the missing generators. It might be possible to use
the ``Selmer group Chabauty'' method of~\cite{Stoll-preprint} in this case, however.


\section{Application to rational diophantine quintuples} \label{S:tuples}

Recall that a \emph{rational diophantine $m$-tuple} is an $m$-tuple $(a_1,\ldots,a_m)$
of distinct nonzero rational numbers such that $a_i a_j + 1$ is a square for all $1 \le i < j \le m$.

Assume that a rational diophantine quadruple $(a_1, a_2, a_3, a_4)$
is given. Then a rational number~$z$ extending it to a rational
diophantine quintuple must satisfy the equations
\[ a_1 z + 1 = u_1^2\,, \quad a_2 z + 1 = u_2^2\,, \quad
   a_3 z + 1 = u_3^2\,, \quad a_4 z + 1 = u_4^2
\]
for suitable rational numbers $u_1, \ldots, u_4$. To homogenize, add
another variable $u_0$ and set $u_0 = 1$ and $a_0 = 0$. Then we have
five equations $a_j z + 1 = u_j^2$.
Eliminating $1$ and~$z$ then results in a diagonal
genus~$5$ curve~$C$, and so we can hope to apply the procedure outlined
in the previous section to find all possible extensions.
Note that this curve~$C$ is the locus of points $(u_0 : \ldots : u_4) \in \PP^4$
such that
\begin{equation} \label{E:rk2}
  \rk \begin{pmatrix} 1 & 1 & 1 & 1 & 1 \\
                      a_0 & a_1 & a_2 & a_3 & a_4 \\
                      u_0^2 & u_1^2 & u_2^2 & u_3^2 & u_4^2
       \end{pmatrix}
       \le 2 \,.
\end{equation}
The ten quadrics of rank~$3$ containing~$C$ are then given by the
$3 \times 3$-minors of the matrix in~\eqref{E:rk2}:
\[ (a_k-a_j) u_i^2 + (a_i-a_k) u_j^2 + (a_j-a_i) u_k^2 = 0
   \qquad \text{for $0 \le i < j < k \le 4$.}
\]

We note that $C$ always has the point $P_0 = (1 : 1 : 1 : 1 : 1)$
(and the points in its orbit under~$A$), which corresponds to
the `illegal' extension by $z = 0$. There are also two further orbits
of points corresponding to
\begin{equation} \label{E:regext}
  z = \frac{(a_1 + a_2 + a_3 + a_4) (a_1 a_2 a_3 a_4 + 1)
               + 2 (a_1 a_2 a_3 + a_1 a_2 a_4 + a_1 a_3 a_4 + a_2 a_3 a_4) \pm 2 s}%
            {(a_1 a_2 a_3 a_4 - 1)^2} \,,
\end{equation}
where $s = \sqrt{(a_1 a_2 + 1)(a_1 a_3 + 1)(a_1 a_4 + 1)(a_2 a_3 + 1)(a_2 a_4 + 1)(a_3 a_4 + 1)}$;
compare~\cite{Dujella1997}*{Theorem~1}.
Generically these three orbits are distinct, but it is possible
that one of the latter two gives $z = 0$. This is always the case
for the one-parameter family
\begin{equation} \label{E:1par}
  \bigl(t-1, t+1, 4 t, 4 t (2 t - 1) (2 t + 1)\bigr)
\end{equation}
of diophantine quadruples. In any case, it is an interesting question
whether for any given rational diophantine quadruple these are the only
possible extensions to a rational diophantine quintuple.

As in Section~\ref{S:genus5}, we use $\pi_i(P_0)$ as the origin on~$F_i$.
Let $(i,j,k,l,m)$ be a permutation of~$(0,1,2,3,4)$. Then an equation
for the elliptic curve~$E_i$ is given by
\[ E_i \colon y^2 = (x + a_j a_k + a_l a_m)(x + a_j a_l + a_k a_m)(x + a_j a_m + a_k a_l) \]
and the $x$-coordinate of the image of $P \in C$ is given by
the quotient of linear forms
\[ x = \frac{N_i(u_0, \ldots, u_4)}{D_i(u_0, \ldots, u_4)} \]
with
\begin{align*}
  N_i(u_0, \ldots, u_4) &= \left|\begin{matrix} 1 & 1 & 1 & 1 \\
                                  a_j & a_k & a_l & a_m \\
                                  a_j^2 & a_k^2 & a_l^2 & a_m^2 \\
                                  a_j b_j u_j & a_k b_k u_k & a_l b_l u_l & a_m b_m u_m
                                 \end{matrix}\right| \qquad\text{and}\\
  D_i(u_0, \ldots, u_4) &= \left|\begin{matrix} 1 & 1 & 1 & 1 \\
                                  a_j & a_k & a_l & a_m \\
                                  a_j^2 & a_k^2 & a_l^2 & a_m^2 \\
                                  u_j & u_k & u_l & u_m
                                 \end{matrix}\right| ,
\end{align*}
where $b_j = a_j - a_k - a_l - a_m$ and similarly for $b_k$, $b_l$, $b_m$.
Then
\[ x + a_j a_k + a_l a_m
     = \frac{V}{D_i(u_0, \ldots, u_4)}
           \Bigl(\frac{u_j-u_k}{a_j-a_k} + \frac{u_l-u_m}{a_l-a_m}\Bigr)
\]
is one of the expressions whose square class gives a component of the
map to~$H$. Here $V$ is the Vandermonde determinant of~$(a_j,a_k,a_l,a_m)$.
An expression for the $y$-coordinate is
\[ y = \frac{V \cdot \bigl((a_j-a_k)(u_j u_k - u_l u_m) + (a_k-a_l)(u_k u_l - u_j u_m)
                            + (a_l-a_j)(u_l u_j - u_k u_m)\bigr)}%
            {D_i(u_0, \ldots, u_4) \cdot \bigl((a_j-a_k) u_l + (a_k-a_l) u_j + (a_l-a_j) u_k\bigr)} \,.
\]

We have implemented the algorithm of Section~\ref{S:genus5} in Magma~\cite{Magma};
see {\sf diophtuples.magma} at~\cite{programs}.
For Step~4b we check if the subgroup generated by points coming from the known
rational points on~$C$ reaches the upper bound for the rank. We then have
generators of a finite-index subgroup and can directly perform the Elliptic
Curve Chabauty computation. When there is a gap between the rank of the known
subgroup and the upper bound, then we would have to find additional generators.
(The most common case is that the rank bound is~$3$, so in principle,
Elliptic Curve Chabauty is possible, but the known subgroup has rank~$2$.
Standard conjectures imply that the rank is then~$3$, so we are missing
one generator.) This can be quite hard for the curves showing up in the
computation, so we have treated this `gap' case as a failure for simplicity.

We have then used our implementation on several rational diophantine quadruples
taken from the one-parameter family~\eqref{E:1par}, with $t \in \Q$ positive
and of small height.
Excluding the cases $t = 1$, $\tfrac{1}{2}$, $\tfrac{1}{3}$ and~$\tfrac{1}{4}$,
which give degenerate quadruples, we were able to show for
\[ t \in \bigl\{ 2, 3, \tfrac{2}{3}, \tfrac{3}{2}, 4, \tfrac{3}{4}, \tfrac{4}{3},
                 5, \tfrac{1}{5}, \tfrac{2}{5}, \tfrac{3}{5}, \tfrac{5}{4}, \tfrac{4}{5} \bigr\} \]
that the extension $z \neq 0$ given by~\eqref{E:regext} is the only possibility.
The case $t = 2$ is Fermat's quadruple; this case is dealt with in detail
(and using a variant of the approach from Section~\ref{S:genus5})
in Section~\ref{S:Fermat} below. When $t = \tfrac{3}{5}$, there is another
`illegal' extension by $z = \tfrac{12}{5}$ (which is already in the original
quadruple); this is because $\bigl(\tfrac{12}{5}\bigr)^2 + 1 = \bigl(\tfrac{13}{5}\bigr)^2$.

In the following table, we give some more detailed information.
This can be checked using file {\sf diophtuples-verify.magma} at~\cite{programs}.

\[ \renewcommand{\arraystretch}{1.2}
   \begin{array}{|c|c|c|c|c|} \hline
     t & T & \text{2-isog.~curve} & \text{field} & \text{subject to} \\ \hline \hline
     2 & [D_{12} + D_{13}] - L & \text{no} & \Q(\sqrt{-6},\sqrt{-14},\sqrt{26}) & \\
     3 & [D_{01} + D_{02}] - L & \text{no} & \Q(\sqrt{10},\sqrt{13},\sqrt{418}) & \\
     \tfrac{2}{3} & [D_{10} + D_{13}] - L & \text{no} & \Q(\sqrt{3},\sqrt{-5},\sqrt{-14}) & \\
     \tfrac{3}{2} & [D_{12} + D_{13}] - L & \text{no} & \Q(\sqrt{-7},\sqrt{65},\sqrt{165}) & \\
     4 & [D_{14} + D_{23}] - L & \text{yes} & \Q(\sqrt{-55},\sqrt{-335},\sqrt{-465}) & \text{GRH} \\
     \tfrac{3}{4} & [D_{10} + D_{34}] - L & \text{no} & \Q(\sqrt{-1},\sqrt{2},\sqrt{15}) & \\
     \tfrac{4}{3} & [D_{02} + D_{13}] - L & \text{yes} & \Q(\sqrt{33},\sqrt{105},\sqrt{6097}) & \text{GRH} \\
     5 & [D_{12} + D_{30}] - L & \text{no} & \Q(\sqrt{-2},\sqrt{-35},\sqrt{247}) & \\
     \tfrac{1}{5} & [D_{04} + D_{31}] - L & \text{yes} & \Q(\sqrt{-1},\sqrt{3},\sqrt{14}) & \\
     \tfrac{2}{5} & [D_{13} + D_{20}] - L & \text{no} & \Q(\sqrt{-7},\sqrt{33},\sqrt{-247}) & \\
     \tfrac{3}{5} & [D_{13} + D_{24}] - L & \text{no} & \Q(\sqrt{7},\sqrt{-26},\sqrt{334}) & \\
     \tfrac{5}{4} & [D_{13} + D_{24}] - L & \text{yes} & \Q(\sqrt{-6},\sqrt{-19},\sqrt{170}) & \\
     \tfrac{4}{5} & [D_{01} + D_{03}] - L & \text{no} & \Q(\sqrt{-2},\sqrt{-35},\sqrt{-3245}) & \text{GRH} \\
     \hline
   \end{array}
\]

The column ``$T$'' gives the representation of the point $T \in G$ that was used to
produce the curve $F_{T,\xi}$ according to the procedure in the previous section.
In the column ``2-isog.~curve'' we note whether the rank bound was obtained from~$F_{T,\xi}$
itself or from the curve $2$-isogenous to it. The column ``field'' lists the octic
field for which we need class group and unit information during the computation of the
$2$-Selmer group of~$F_{T,\xi}$ or its isogenous curve. Finally, we indicate in the last
column whether the result is conditional on the Generalized Riemann Hypothesis.
The entry ``GRH'' indicates that this assumption was used to speed up the class group
computation for the octic field, which would have been infeasible in reasonable time otherwise.

The rank of the subgroup of~$F_{T,\xi}(K)$ generated by images of known points on~$C$
is~$2$ in all cases except for $t = \tfrac{3}{5}$, where it is~$3$. The upper bound
deduced from the $2$-Selmer group coincides with this lower bound in all cases.
Except again for $t = \tfrac{3}{5}$, there is only one class~$\xi$ to consider; in the
exceptional case, there are two (recall that we get the extra `illegal' extension
by~$\tfrac{12}{5}$, so there are additional rational points on~$C$).


\section{Possible extensions of Fermat's quadruple} \label{S:Fermat}

Fermat discovered the diophantine quadruple $(1,3,8,120)$ and Euler
found that it can be extended to a rational diophantine quintuple
by adding the number~$777480/8288641$. It seems to be an open question
whether this is the only possibility. We now give a fairly detailed proof that this
is indeed the case. (Note that the statement is contained in the results obtained
in the previous section; it is the case $t = 2$. The proof given here is slightly
different, though.) We do this using a variant of the approach
described above, which replaces Step~2 by the computation of the
fake $2$-Selmer set of one of the genus~$2$ curves arising as
quotients of the genus~$5$ curve.

\begin{theorem} \label{T:Fermat}
  The only way of extending Fermat's diophantine quadruple $(1,3,8,120)$
  to a rational diophantine quintuple $(1,3,8,120,z)$ is to take
  \[ z = \frac{777480}{8288641} \,. \]
\end{theorem}

\begin{proof}
  We want to determine all nonzero rational~$z$ satisfying the following
  system of equations (with suitable $u_1, u_2, u_3, u_4 \in \Q$):
  \[ z + 1 = u_1^2, \quad 3z + 1 = u_2^2, \quad 8z + 1 = u_3^2, \quad 120z + 1 = u_4^2. \]
  Writing $x = u_4$, we see that
  \[ x^2 + 119 = 120 u_1^2, \quad x^2 + 39 = 40 u_2^2, \quad x^2 + 14 = 15 u_3^2, \]
  so
  \[ 5 (x^2 + 119) (x^2 + 39) (x^2 + 14) = (600 u_1 u_2 u_3)^2, \]
  and we obtain a rational point on the hyperelliptic curve of genus~$2$
  \[ H \colon y^2 = 5 (x^2 + 119) (x^2 + 39) (x^2 + 14). \]
  A quick search finds the points
  \[ (\pm 1, \pm 600) \qquad\text{and}\qquad
    \Bigl(\pm\frac{10079}{2879}, \pm\frac{22426285104600}{2879^3}\Bigr).
  \]
  The first quadruple of points corresponds to the degenerate solution $z = 0$,
  the second one gives Euler's solution. We must show that these are the
  only rational points on~$H$.

  We perform a `two-cover descent' on~$H$ as in~\cite{BS2009}. This results
  in a two-element `fake 2-Selmer set', whose elements are accounted
  for by the `trivial' points in~$H(\Q)$ with $x = \pm 1$. In particular,
  we see that the automorphism of~$H$ given by changing the sign of~$x$
  interchanges the two Selmer set elements, so it suffices to find all
  rational points mapping to one of them, say the element corresponding
  to $x = 1$. Any such rational point will give rise to a point
  over $K = \Q(\sqrt{-119}, \sqrt{-39})$ on the curve
  \[ E \colon Y^2 = 15 (1 - \sqrt{-119}) (1 - \sqrt{-39}) \cdot
                    (X^2 + 14) (X - \sqrt{-119}) (X - \sqrt{-39})
  \]
  with $X$-coordinate in~$\Q$ (since $X = x$). This curve~$E$ is an
  elliptic curve over~$K$; it has one $K$-rational point of order~$2$.
  We compute its $2$-Selmer group and find that it is isomorphic to~$(\Z/2\Z)^3$.
  This implies that the rank of~$E(K)$ is at most~$2$. We find two independent
  points in~$E(K)$ (coming from the known points on~$H$), so the rank
  is indeed~$2$. We saturate the group generated by the $2$-torsion point
  and these two points at $2$, $3$, $5$, $7$, which is enough for
  the Elliptic Curve Chabauty computation as implemented
  in Magma~\cite{Magma}. This computation finally shows that the only
  points in~$E(K)$ with $X \in \Q$ are the points with $X = 1$
  or $X = 10079/2879$. This concludes the proof.
\end{proof}

\begin{remark}
  Since the curve~$H$ in the proof is bi-elliptic, with both elliptic curve
  quotients of rank~$1$, it would also possible to use ``Quadratic Chabauty''
  to determine its set of rational points. See recent work by Balakrishnan
  and~Dogra~\cites{BD1,BD2}.
\end{remark}


\section{Mordell-Weil groups of elliptic curves over $\Q(t)$} \label{S:EQt}

We now consider the problem of determining generators for the group $E(\Q(t))$
of $\Q(t)$-rational points on an elliptic curve~$E$ defined over~$\Q(t)$,
in the case when all points of order~$2$ on~$E$ are defined over~$\Q(t)$.
The usual $2$-descent approach gives an embedding
\[ \delta \colon \frac{E(\Q(t))}{2 E(\Q(t))} \To H_0
          \subset \frac{\Q(t)^\times}{\Q(t)^{\times 2}} \times \frac{\Q(t)^\times}{\Q(t)^{\times 2}}
\]
with a suitable finite subgroup~$H_0$. For all but finitely many $\tau \in \Q$,
we can specialize~$E$ to an elliptic curve~$E_\tau$ over~$\Q$, and we have
a natural homomorphism $\rho_\tau \colon E(\Q(t)) \to E_\tau(\Q)$. By Silverman's specialization
theorem~\cite{Silverman1983}, $\rho_\tau$ is injective once $\tau$
has sufficiently large height. If we can find such a $\tau$ with the property
that the subgroup of~$E(\Q(t))$ generated by some known points surjects onto~$E_\tau(\Q)$,
then this proves that our known points already generate~$E(\Q(t))$.
The lower bound for the height of~$\tau$ that one can extract from Silverman's approach
tends to be too large to be practical; also, $\rho_\tau$ is usually injective
also for most ``small''~$\tau$. So it is useful to have a more concrete computational
criterion for testing the injectivity of~$\rho_\tau$. Such a criterion was provided
by Gusi\'c and~Tadi\'c in a recent paper~\cite{GusicTadic2015} (building on the
earlier paper~\cite{GusicTadic2012}), in the case when
$E$ has at least one point of order~$2$ defined over~$\Q(t)$. We improve on their
approach somewhat by making use of several specializations to cut down the group~$H_0$
that contains the image of~$E(\Q(t))/2E(\Q(t))$, which gives the criterion a better
chance of success.

We then use our method to determine generators of the Mordell-Weil group
of the elliptic curve
\[ E \colon y^2 = \bigl(x + 4t(t-1)\bigr) \bigl(x + 4t(t+1)\bigr) \bigl(x + (t-1)(t+1)\bigr) \]
over~$\Q(t)$. It turns out that the group has rank~$1$ and maps isomorphically
under~$\rho_2$ to~$E_2(\Q)$. We then use this result to show that the ``regular''
extension of a family of rational diophantine quadruples to a rational diophantine
quintuple is the only generic (i.e., given by a rational function in the parameter)
such extension; see Sections \ref{S:ExE} and~\ref{S:Generic} below.

\medskip

So let $E$ be an elliptic curve over the rational function field~$\Q(t)$.
Since $\Q(t)$ is a finitely generated field, the Mordell-Weil group~$E(\Q(t))$
is a finitely generated abelian group~\cite{LangNeron1959}. In the following, we will
assume that all points of order~$2$ on~$E$ are defined over~$\Q(t)$
and that $E$ is given by a Weierstrass equation with coefficients in~$\Z[t]$
of the form
\[ E \colon y^2 = \bigl(x - e_1(t)\bigr) \bigl(x - e_2(t)\bigr) \bigl(x - e_3(t)\bigr) \,. \]
As over any field (of characteristic~$\neq 2$), we have the exact sequence
\[ 0 \To E(\Q(t))[2] \To E(\Q(t)) \stackrel{\cdot 2}{\To} E(\Q(t)) \stackrel{\delta}{\To} H \,, \]
where $H$ is the subgroup of $\bigl(\Q(t)^\times/\Q(t)^{\times 2}\bigr)^3$
consisting of triples such that the product of the three entries is trivial.
One can show in the usual way that the image of~$\delta$ is
contained in the subgroup generated by the prime divisors in the UFD~$\Z[t]$
of
\[ \Delta(t) = (e_1(t) - e_2(t))(e_1(t) - e_3(t))(e_2(t) - e_3(t)) \in \Z[t] \]
in each of the three components; see~\cite{GusicTadic2015}.

Let $\tau \in \Q$ be such that $\Delta(\tau) \neq 0$. Then we can specialize~$E$
to an elliptic curve~$E_\tau$ over~$\Q$, and we obtain a specialization
homomorphism $\rho_\tau \colon E(\Q(t)) \to E_\tau(\Q)$.
Gusi\'c and Tadi\'c in~\cite{GusicTadic2015} give a criterion for when
$\rho_\tau$ is injective. In the following, we will give a streamlined proof
of their result (which is based on the same ideas), which we will then use to
devise a method that can show that a known set of points in~$E(\Q(t))$
generates the latter group.

We begin with an easy lemma.

\begin{lemma} \label{L:diagram}
  Assume that we have the following commutative diagram of abelian groups
  with exact rows:
  \[ \xymatrix{A \ar@{->>}[d]_{\alpha} \ar[r] & B \ar[d]_{\beta} \ar[r]^{\varphi}
                 & C \ar[d]^{\gamma} \ar[r] & D \ar@{^(->}[d]^{\delta} \\
               A' \ar[r] & B' \ar[r] & C' \ar[r] & D'
              }
  \]
  Assume further that $\alpha$ is surjective and $\delta$ is injective. Then $\varphi$
  induces a surjective homomorphism $\ker \beta \to \ker \gamma$.
\end{lemma}

\begin{proof}
  This is an easy diagram chase.
\end{proof}

\begin{corollary} \label{C:inj}
  Let $E$ and~$H$ be as above. Assume that $H' \subset H$ is finitely
  generated and contains~$\delta\bigl(E(\Q(t))\bigr)$. If the specialization
  homomorphism $h_\tau \colon H' \to H_\tau = (\Q^\times/\Q^{\times 2})^3$
  associated to~$\tau$ is injective,
  then the homomorphism $\rho_\tau \colon E(\Q(t)) \to E_\tau(\Q)$ is also injective.
\end{corollary}

\begin{proof}
  We first remark that $\rho_\tau$ is injective on torsion. Let $K_\tau = \ker \rho_\tau$;
  then $K_\tau$ is a finitely generated torsion-free abelian group, so $K_\tau$ is
  a free abelian group.
  We consider the commutative diagram with exact rows
  \[ \xymatrix{E(\Q(t))[2] \ar[r] \ar[d]_{\simeq}
                 & E(\Q(t)) \ar[r]^{\cdot 2} \ar[d]_{\rho_\tau}
                 & E(\Q(t)) \ar[r]^-{\delta} \ar[d]_{\rho_\tau}
                 & H' \ar[d]^{h_\tau} \\
               E_\tau(\Q)[2] \ar[r]
                 & E_\tau(\Q) \ar[r]^{\cdot 2}
                 & E_\tau(\Q) \ar[r]^-{\delta_\tau}
                 & H_\tau\,.
              }
  \]
  Since $E$ has full $2$-torsion over~$\Q(t)$, the leftmost map is an isomorphism.
  By assumption the rightmost map is injective. So Lemma~\ref{L:diagram} tells us that
  $2K_\tau = K_\tau$. Since $K_\tau$ is free, this implies $K_\tau = 0$ as desired.
\end{proof}

This easily implies Theorem~1.1 of~\cite{GusicTadic2015}: their condition
is equivalent to the injectivity of $H' \to H_\tau$, where $H'$ is a slightly
refined version of the general upper bound for~$\delta\bigl(E(\Q(t))\bigr)$
mentioned above.

\begin{remark} \label{R:2isog}
  If $E(\Q(t))[2] \simeq \Z/2\Z$ and $E'$ is the $2$-isogenous curve with
  dual isogenies $\phi \colon E \to E'$ and $\hat{\phi} \colon E' \to E$,
  then we can use
  \[ \xymatrix{E(\Q(t))[2] \times E'(\Q(t))[2] \ar[r] \ar[d]
                 & E(\Q(t)) \times E'(\Q(t)) \ar[r]^{\Phi} \ar[d]_{\rho_\tau}
                 & E(\Q(t)) \times E'(\Q(t)) \ar[r]^-{\delta} \ar[d]_{\rho_\tau}
                 & H' \ar[d]_{h_\tau} \\
               E_\tau(\Q)[2] \times E'_\tau(\Q)[2] \ar[r]
                 & E_\tau(\Q) \times E'_\tau(\Q) \ar[r]^{\Phi}
                 & E_\tau(\Q) \times E'_\tau(\Q) \ar[r]^-{\delta_\tau}
                 & H_\tau
              }
  \]
  with $\Phi(P, Q) = (\hat{\phi}(Q), \phi(P))$ and $H' \subset (\Q(t)^\times/\Q(t)^{\times 2})^2$
  coming from the descent maps associated to $\phi$ and~$\hat{\phi}$.
  Note that $\Phi \circ \Phi$ is multiplication by~$2$ on~$E \times E'$.
  If the leftmost map is surjective and $h_\tau$ is injective, then both specialization
  maps $E(\Q(t)) \to E_\tau(\Q)$ and $E'(\Q(t)) \to E'_\tau(\Q)$ are injective.
  This recovers Theorem~1.3 of~\cite{GusicTadic2015}; note that their condition
  implies that there are no additional rational $2$-torsion points on the specialized curves.
\end{remark}

\begin{remark}
  The exact same result clearly holds when we replace $E$ by the Jacobian variety~$J$
  of a hyperelliptic curve with the property that all $2$-torsion points on~$J$
  are $\Q(t)$-rational. Remark~\ref{R:2isog} generalizes in a similar way to Jacobians
  of curves of genus~$2$ admitting a Richelot isogeny whose kernel consists of
  points defined over~$\Q(t)$.
\end{remark}

\begin{remark}
  If $E(\Q(t))[2] = 0$, then one can still use this approach in principle (one needs
  to use specializations with $E_\tau(\Q)[2] = 0$); however, the difficulty lies in
  obtaining a suitable bounding group~$H'$. It is a subgroup of the group of square classes
  in the function field of the curve given by the $2$-torsion sections in the elliptic
  surface associated to~$E$, and so it will very likely be necessary to obtain information
  on the $2$-torsion in the Picard group of this curve (which is trigonal, but can be
  arbitrarily complicated in any other way).
\end{remark}

If one wants to use the criterion of Gusi\'c and~Tadi\'c directly, one needs
to find~$\tau$ such that $-1$ and the distinct prime factors of~$\Delta(t)$
(in the UFD~$\Z[t]$ and up to sign) have
independent images in~$\Q^\times/\Q^{\times 2}$. If there are many prime factors,
this usually means that $\tau$ cannot be taken to have small height.
This in turn may lead to difficulties when trying to determine the rank of~$E_\tau(\Q)$.

We suggest the following modified approach. Let $P_1, \ldots, P_r \in E(\Q(t))$ be
known independent points. We want to show that they generate the free part of~$E(\Q(t))$.

\begin{enumerate}[1.]
  \item Find the prime divisors of~$\Delta$ in~$\Z[t]$ and let $H_0$ be the subgroup of~$H$
        consisting of triples all of whose entries are represented by a divisor of~$\Delta$.
        (Then the image of~$\delta$ is contained in~$H_0$.)
  \item For a finite set~$T$ of $\tau \in \Q$ such that $\Delta(\tau) \neq 0$,
        compute the $2$-Selmer group of~$E_\tau$ as a subgroup $S_\tau \subset H_\tau$
        (if feasible, compute tighter bounds for $\delta_\tau(E_\tau(\Q))$, for example
        by determining~$E_\tau(\Q)$).
  \item Set $H' \colonequals \{a \in H_0 : \forall \tau \in T \colon h_\tau(a) \in S_\tau\}$.
  \item Now consider values of~$\tau$ such that $\Delta(\tau) \neq 0$ and
        $H' \cap \ker h_\tau = 0$.
        If $\rho_\tau(P_1), \ldots, \rho_\tau(P_r)$ generate the free part of~$E_\tau(\Q)$,
        then $P_1, \ldots, P_r$ generate the free part of~$E(\Q(t))$.
\end{enumerate}

To see that Step~4 works, first note that by Corollary~\ref{C:inj}, we know
that $\rho_\tau$ is injective. If the known subgroup of the free part of~$E(\Q(t))$
surjects under~$\rho_\tau$ onto the free part of~$E_\tau(\Q)$, then it follows
that the known subgroup must already be all of the free part.

We have implemented this procedure in Magna~\cite{Magma}; see the file {\sf ellQt.magma}
at~\cite{programs}.


\section{Examples} \label{S:ExE}

In the following, we will use additive notation for the group~$H$ and its
subgroups. Note that they are killed by~$2$, so they can be considered as
vector spaces over~$\F_2$.

We consider the curve
\[ E \colon y^2 = \bigl(x + 4t(t-1)\bigr) \bigl(x + 4t(t+1)\bigr) \bigl(x + (t-1)(t+1)\bigr) \,. \]
It has the point
\[ P = \bigl(0, 4t(t-1)(t+1)\bigr) \]
of infinite order. The discriminant factors as
\[ \Delta_E = -2^{10} t^2 (t-1)^2 (t+1)^2 (3t-1)^2 (3t+1)^2 \,, \]
hence the image of~$E(\Q(t))$ under~$\delta$ is contained in~$H_0$, where
\[ H_0 = H \cap \langle -1, 2, t, t-1, t+1, 3t-1, 3t+1 \rangle^3 \,, \]
so that $H_0$ has dimension~$14$.
The image of the known subgroup of~$E(\Q(t))$ is~$H_1$, generated by
\begin{align*}
   &\bigl(-2t(t-1)(3t-1), 2t, -(t-1)(3t-1)\bigr)\,, \\
   &\bigl(-2t, 2t(t+1)(3t+1), -(t+1)(3t+1)\bigr) \qquad\text{and} \\
   &\bigl(t(t-1), t(t+1), (t-1)(t+1)\bigr)\,.
\end{align*}
We can check that for the specializations~$E_t$ with $t = 2$, $t = 3$ and~$t = 5$,
the groups $E_t(\Q)$ are generated by the specializations of the known generators
of~$E(\Q(t))$. (The Magma function \texttt{MordellWeilGroup} does this for us.)
This implies that
\[ H_1 \subset \delta(E(\Q(t))) \subset
     \bigl(H_1 + \ker(h_2)\bigr) \cap \bigl(H_1 + \ker(h_3)\bigr) \cap \bigl(H_1 + \ker(h_5)\bigr) \,.
\]
We can easily check that the intersection on the right equals~$H_1$,
so that $\delta(E(\Q(t))) = H_1$ (this already shows that $E(\Q(t))$
has rank~$1$ and that the known points generate a subgroup of finite odd index).

For $\tau = 2$ (and also $\tau = 3$ and $\tau = 5$),
we have that $H_1$ meets the kernel of~$h_\tau$ trivially.
Since the specializations of the known generators of~$E(\Q(t))$ are generators of~$E_\tau(\Q)$,
we have shown that the known points actually generate the full group~$E(\Q(t))$:

\begin{proposition} \label{P:EQt}
  The group $E(\Q(t))$ is isomorphic to $\Z/2\Z \times \Z/2\Z \times \Z$, with generators
  \[ T_1 = \bigl(-4t(t-1), 0\bigr)\,, \quad T_2 = \bigl(-4t(t+1), 0\bigr) \quad\text{and}\quad
     P = (0, 4t(t-1)(t+1)\bigr) \,.
  \]
\end{proposition}

\begin{remark}
  The smooth elliptic surface~$\calE$ over~$\PP^1$ associated to~$E$ is a rational
  surface. It has six bad fibers (over $t = 0, 1, -1, \frac{1}{3}, -\frac{1}{3}$ and~$\infty$),
  which are all of type~$\operatorname{I}_2$. By the Shioda-Tate formula~\cite{Shioda1999}, this gives
  \[ \rank E(\bar{\Q}(t)) = 10 - 2 - 6 = 2 \,, \]
  so there must be another point defined over~$K(t)$ for some finite extension~$K$ of~$\Q$.
  After searching a bit, one finds the independent point
  \[ Q = \bigl(2t(t+1), 2\sqrt{3} t(t+1)(3t-1)\bigr) \,, \]
  which is defined over~$\Q(\sqrt{3})$ (and actually comes from the quadratic twist
  of~$E$ by~$3$). The height pairing matrix of $P$ and~$Q$ is the diagonal matrix
  with entries $\tfrac{1}{2}$, $\tfrac{1}{2}$. Since the bad fibers are of type~$\operatorname{I}_2$,
  it follows that the canonical height of any point is in~$\tfrac{1}{2}\Z$.
  Since there is no proper superlattice of the square lattice generated by $P$ and~$Q$
  whose points satisfy this condition,
  this implies that $P$ and~$Q$ generate the free part of~$E(\bar{\Q}(t))$.
  We thus obtain an alternative proof of Proposition~\ref{P:EQt}.
  We would also like to mention that Dujella~\cite{Dujella2000}*{Theorem~4} had
  already shown that $E(\Q(t))$ has rank~$1$.
\end{remark}

In the same way, one can show the following.

\begin{proposition} \strut \label{P:otherE}
  \begin{enumerate}[\upshape(1)]
    \item The group of $\Q(t)$-points on the elliptic curve
          \[ y^2 = \bigl(x + (t^2-1)\bigr)\bigl(x + 4t(t-1)(4t^2-1)\bigr)\bigl(x + 4t(t+1)(4t^2-1)\bigr) \]
          is isomorphic to $\Z/2\Z \times \Z/2\Z \times \Z^2$, generated by the points
          \begin{gather*}
             \bigl(-(t^2-1), 0\bigr), \quad \bigl(-4t(t-1)(4t^2-1), 0\bigr), \quad
             \bigl(0, 4t(t^2-1)(4t^2-1)\bigr) \\
             \text{and}\quad
             \bigl(-8t^2(2t^2-1), 4t(4t^2-t-1)(4t^2+t-1)\bigr)\,.
          \end{gather*}
    \item The group of $\Q(t)$-points on the elliptic curve
          \[ y^2 = \bigl(x + 4t(t-1)\bigr)\bigl(x + 4t(t-1)(4t^2-1)\bigr)\bigl(x + 16t^2(4t^2-1)\bigr) \]
          is isomorphic to $\Z/2\Z \times \Z/2\Z \times \Z^2$, generated by the points
          \begin{gather*}
             \bigl(-4t(t-1), 0\bigr), \quad \bigl(-4t(t-1)(4t^2-1), 0\bigr), \quad
             \bigl(0, 16t^2(t-1)(4t^2-1)\bigr) \\
             \text{and}\quad
             \bigl(4t(t-1)(2t-1)(4t+1), 16t^2(t-1)(2t-1)(3t+1)(4t-1)\bigr)\,.
          \end{gather*}
    \item The group of $\Q(t)$-points on the elliptic curve
          \[ y^2 = \bigl(x + 4t(t+1)\bigr)\bigl(x + 4t(t+1)(4t^2-1)\bigr)\bigl(x + 16t^2(4t^2-1)\bigr) \]
          is isomorphic to $\Z/2\Z \times \Z/2\Z \times \Z^2$, generated by the points
          \begin{gather*}
             \bigl(-4t(t+1), 0\bigr), \quad \bigl(-4t(t+1)(4t^2-1), 0\bigr), \quad
             \bigl(0, 16t^2(t+1)(4t^2-1)\bigr) \\
             \text{and}\quad
             \bigl(4t(t+1)(2t+1)(4t-1), 16t^2(t+1)(2t+1)(3t-1)(4t+1)\bigr)\,.
          \end{gather*}
    \item The group of $\Q(t)$-points on the elliptic curve
          \[ y^2 = x\bigl(x + 16t(2t^2-1)\bigr)\bigl(x + (3t+1)(4t+1)(4t^2-t-1)\bigr) \]
          is isomorphic to $\Z/2\Z \times \Z/2\Z \times \Z$, generated by the points
          \begin{gather*}
             \bigl(0, 0\bigr), \; \bigl(-16t(2t^2-1), 0\bigr)
             \quad\text{and}\quad
             \bigl(2(3t+1), 2(3t+1)(4t-1)(4t^2+t-1)\bigr)\,.
          \end{gather*}
  \end{enumerate}
\end{proposition}

\begin{proof}
  Running our implementation on the four given curves and the
  given points results in a proof that the points generate the group in each case.
  See the commented-out last part of~{\sf ellQt.magma} for the relevant commands.
\end{proof}

We note that the specializations at integers of the curve considered in part~(1)
have been studied by Najman~\cite{Najman2009}.

Further applications of the method can be found in~\cites{DGT2015,DPT2016}.


\section{Generic extensions of the family of diophantine quadruples} \label{S:Generic}

As already mentioned, there is a one-parameter family of rational diophantine quadruples,
given by
\[ (a_1, a_2, a_3, a_4) = \bigl(t-1,t+1,4t,4t(4t^2-1)\bigr) \quad
    \text{with $t \in \Q \setminus \bigl\{-1,-\tfrac{1}{2},-\tfrac{1}{3},-\tfrac{1}{4},
                                          0,\tfrac{1}{4},\tfrac{1}{3},\tfrac{1}{2},1\bigr\}$.}
\]
We have seen in Section~\ref{S:tuples} that it can be extended to a quintuple by setting
\[ a_5 = \frac{4 t (2t - 1) (2t + 1) (4t^2 - 2t - 1) (4t^2 + 2t - 1) (8t^2 - 1)}%
              {(64t^6 - 80t^4 + 16t^2 - 1)^2}
\]

We can now use Proposition~\ref{P:EQt} to show that this extension is in fact the
only \emph{generic} extension of the quadruple, in the sense that there is no
other rational function $a_5(t) \in \Q(t)^\times$ such that $(a_1(t), \ldots, a_5(t))$
is a diophantine quintuple in~$\Q(t)$ (i.e., such that $a_i(t) a_j(t) + 1$
is a square in~$\Q(t)$ for all $1 \le i < j \le 5$):

\begin{theorem} \label{T:ratptsQt}
  The only function $f(t) \in \Q(t)^\times$ such that
  \[ (t-1) f(t) + 1\,, \quad (t+1) f(t) + 1\,, \quad 4 t f(t) + 1 \quad\text{and}\quad
     4t(4t^2-1) f(t) + 1
  \]
  are all squares in~$\Q(t)$ is
  \[ f(t) = \frac{4 t (2t - 1) (2t + 1) (4t^2 - 2t - 1) (4t^2 + 2t - 1) (8t^2 - 1)}%
                 {(64t^6 - 80t^4 + 16t^2 - 1)^2} \,.
  \]
\end{theorem}

\begin{proof}
  Any such $f(t)$ gives rise to a $\Q(t)$-rational point on the curve~$C$ defined by
  \[ (t-1) x + 1 = u_1^2\,, \quad (t+1) x + 1 = u_2^2\,, \quad 4t x + 1 = u_3^2\,, \quad
     4t(4t^2-1) x + 1 = u_4^2
  \]
  with $x = f(t)$. We have a morphism
  \[ \pi \colon C \To E\,, \qquad
     (x, u_1, u_2, u_3, u_4) \longmapsto \bigl(4t(t-1)(t+1) x, 4t(t-1)(t+1) u_1 u_2 u_3\bigr) \,,
  \]
  where $E$ is the elliptic curve considered in Proposition~\ref{P:EQt}.
  The morphism extends to a finite morphism of degree~$8$ from the projective closure of~$C$ to~$E$.
  Now consider the following commutative diagram.
  \[ \xymatrix{C(\Q(t)) \ar[r]^{\pi} \ar[d]^{\rho_{C,2}} & E(\Q(t)) \ar[d]^{\rho_2}_{\simeq} \\
               C_2(\Q) \ar[r]^{\pi_2} & E_2(\Q)
              }
  \]
  Here the subscript~$2$ denotes the specializations at $t = 2$.
  Since the genus of~$C_2$ is $5 \ge 2$, there are only finitely many
  points in~$C_2(\Q)$. In fact, we have shown in Theorem~\ref{T:Fermat}
  that there are precisely $32$~points in~$C_2(\Q)$, corresponding to the
  extension of $(1,3,8,120)$ by (the illegal value)~zero and by~$777480/8288641$
  (recall that each extension gives rise to an $A$-orbit of $16$~points on~$C_2$).
  This implies that the image
  of~$C_2(\Q)$ in~$E_2(\Q)$ consists of $\{\pm P_2, \pm 5P_2\}$ (where $P_2 = (0, 24) \in E_2(\Q)$
  is the specialization of~$P$). Here the points giving the extension by zero map
  to~$\pm P_2$, and the points giving the nontrivial extension map to~$\pm 5P_2$.

  Since the right-hand vertical map is an isomorphism by Proposition~\ref{P:EQt},
  this shows that any point in~$C(\Q(t))$ must map to $\pm P$ or~$\pm 5P$ in~$E(\Q(t))$.
  This leads to $x = 0$ or $x = f(t)$, finishing the proof.
\end{proof}

The curve $E$ and the other four elliptic curves showing up in Proposition~\ref{P:otherE}
are the five elliptic curves occurring as quotients of the genus~$5$ curve~$C$ over~$\Q(t)$
associated to the quadruple $(t-1, t+1, 4t, 4t(4t^2-1))$. So Propositions
\ref{P:EQt} and~\ref{P:otherE} together imply that the group $J(\Q(t))$
has rank~$8$, where $J$ denotes the Jacobian variety of~$C$.
We can actually say more.

\begin{proposition}
  Let $C$ denote the genus~$5$ curve over~$\Q(t)$ associated to the diophantine quadruple
  \[ \bigl(t-1, t+1, 4t, 4t(4t^2-1)\bigr) \]
  over~$\Q(t)$, and denote its Jacobian by~$J$. Then $C$ has exactly~$32$ $\Q(t)$-points, and
  the group $J(\Q(t))$ is isomorphic to~$(\Z/2\Z)^5 \times \Z^8$ and is generated by
  differences of these points.
\end{proposition}

\begin{proof}
  We use the smooth projective model of~$C$ obtained by setting \hbox{$u_0 = 1$} and eliminating~$x$.
  The statement on the $\Q(t)$-points of~$C$ is equivalent to Theorem~\ref{T:ratptsQt}:
  the illegal extension by zero gives rise to the $16$~points $(1 : \pm 1 : \pm 1 : \pm 1 : \pm 1)$,
  and the regular extension by~$f$ gives rise to a second orbit of $16$~points.
  By the theorem, these are all the $\Q(t)$-points on~$C$.

  Let $E_0, \ldots, E_4$ be the elliptic curves obtained by eliminating
  one of the variables $u_0, \ldots, u_4$. They are given as intersections
  of two quadrics in~$\PP^3$; choosing the image of the point $(1:1:1:1:1)$ as the origin,
  we obtain isomorphisms with the curve~$E$ considered above and the four further
  elliptic curves considered in Proposition~\ref{P:otherE} (in a different order).
  Let $G \subset J(\Q(t))$ denote the subgroup generated by the differences of
  the $\Q(t)$-points on~$C$. We find that the image in $E_0 \times \ldots \times E_4$
  of~$G$ is a finite-index subgroup of
  $G' = (E_0 \times \ldots \times E_4)(\Q(t))$ that is isomorphic to~$\Z^8$;
  the index is a power of~$2$. We know that the kernel of this map
  (which is the isogeny~$\phi$) is isomorphic to $(\Z/2\Z)^5$ and is contained
  in the group generated by differences of the $\Q(t)$-points on~$C$.
  (This is because the divisors~$D_{ij}$ considered in Section~\ref{S:genus5}
  can be realized as sums of four $\Q(t)$-points on~$C$, by pulling back
  the image of any such point on~$Q_{ij}$.) It therefore suffices to show that
  $G$ is $2$-saturated. Any element of~$G$ that is divisible by~$2$ maps
  to an element of~$G'$ that is divisible by~$2$. There are $15$ nontrivial
  cosets of the intersection of the image of~$G$ with~$2 G'$ inside the image of~$G$.
  We find a representative $P \in G$ of each of these cosets; it is then sufficient
  to show that no point of the form $P+T$ is divisible by~$2$, where $T$ is any
  element of the kernel.

  Since we can easily verify that an element of the group of $\Q(t)$-rational
  points on the Jacobian of a genus~$2$ curve is not divisible by~$2$ by checking
  that its image under the `$x-T$ map' is nontrivial (see for example~\cite{Stoll2001}),
  we consider some of the genus~$2$ curves~$D$ that occur as quotients of~$C$
  under a subgroup of the automorphism group consisting of sign changes on
  an even number of variables out of some three-element set. (Such a quotient
  was also used above in Section~\ref{S:Fermat}.) We note that the image of $P+T$
  as above must be the sum of the image of~$P$ and a $\Q(t)$-rational $2$-torsion
  point on the Jacobian of~$D$. For each of the $15$ choices of~$P$, we find
  some curve~$D$ such that none of these points is divisible by~$2$ in~$\Jac(D)(\Q(t))$.
  This finishes the proof.
\end{proof}


\end{document}